\documentclass[10pt, oneside]{article}

\usepackage[dvips]{graphicx}
\usepackage{pstricks}
\usepackage{color}

\usepackage{amssymb}
\usepackage{amsmath}
\usepackage{amsthm}
\usepackage{color}

\newtheorem{theorem}{Theorem}[section]

\newtheorem{lemma}[theorem]{Lemma}
\newtheorem{proposition}[theorem]{Proposition}

\theoremstyle{definition}
\newtheorem{definition}{Definition}[section]

\theoremstyle{remark}

\newcommand{\RR}{\mathbb{R}}
\newcommand{\NN}{\mathbb{N}}

\newcommand{\XX}{\mathbb{X}}

\newcommand{\cF}{\mathcal{F}}

\newcommand{\cL}{\mathcal{L}}
\newcommand{\cS}{\mathcal{S}}

\newcommand{\cZ}{\mathcal{Z}}

\newcommand{\e}{\varepsilon}

\begin{document}
\title{Multiplicity of bounded solutions to the $k$-Hessian equation with a Matukuma-type source\thanks{The first author was supported by JSPS KAKENHI Grant Number 16K05225. The second and the third author were partially supported by Fondecyt Grant 1150230.}}

\author{Y. Miyamoto, J. S\'anchez  and V. Vergara}
\date{}
\maketitle

\begin{abstract}
The aim of this paper is to deal with the $k$-Hessian counterpart of the Laplace equation involving a nonlinearity studied by Matukuma. Namely, our model is the problem
\begin{equation*}
(1)\;\;\;\begin{cases}
S_k(D^2u)= \lambda  \frac{|x|^{\mu-2}}{(1+|x|^2)^{\frac{\mu}{2}}} (1-u)^q &\mbox{in }\;\; B,\\
u <0 & \mbox{in }\;\; B,\\
u=0 &\mbox{on }\partial B,
\end{cases}
\end{equation*}
where $B$ denotes the unit ball in $\RR^n,n>2k$ ($k\in\NN$), $\lambda>0$ is an additional parameter, $q>k$ and $\mu\geq 2$. In this setting, through a transformation recently introduced by two of the authors that reduces problem (1) to a non-autonomous two-dimensional generalized Lotka-Volterra system, we prove the existence and multiplicity of solutions for the above problem combining dynamical-systems tools, the intersection number between a regular and a singular solution and the super and subsolution method.
\end{abstract}
2010 Mathematics Subject Classification: primary 35B33; secondary 34C37, 34C20, 35J62, 70K05.

\noindent {\em Keywords:}\, $k$-Hessian operator; Radial solutions; Non-autonomous Lotka-Volterra system; Phase analysis; Critical exponents; 
Singular solution; Intersection number.
{\footnotesize}

\section{Introduction}
The classical Matukuma equation
\begin{equation*}
\Delta u+\frac{u^q}{1+|x|^2}=0\;\; \mbox{in}\;\; \RR^3,
\end{equation*}
was proposed by T. Matukuma \cite{Matukuma30} as a mathematical model for a globular cluster of stars, where $q>1$ is a parameter and $u>0$ stands for the gravitational potential. This equation has been extensively studied in the literature, see e.g. \cite{SaQu10, WaZL12, BFaHo86, FeQT09, Li93, MoYY05}.

A more general model was proposed by J. Batt, W. Faltenbacher and E. Horst \cite{BFaHo86} which contains as a particular case the equation
\begin{equation*}
\Delta u+\frac{|x|^{\mu-2}}{(1+|x|^2)^\frac{\mu}{2}}u^q=0\;\; \mbox{in}\;\; \RR^3,
\end{equation*}
where $\mu>0$ is an additional parameter, see e.g. \cite{BFaHo86, Li93} and the references therein. Recently, an extensive study of solutions to the preceding equation has been extended to higher $(n>3)$ dimensions, see \cite{WaZhLi12}. It is well-known that this kind of equation admits three different types of positive radial solutions depending on the parameters $\mu$, $q$ and $n$ (in case $n>3$). In particular, they admit the so-called $E$-solutions which are characterized by $\lim_{r\to 0}u(r)<\infty$, see \cite{BatLi10, WaZhLi12}.

\medbreak

The aim of this paper is to study radially symmetric bounded solutions to the Matukuma equation in the framework of the $k$-Hessian operator. More precisely, we consider the question of the existence and multiplicity of radially symmetric bounded solutions of the problem
\begin{equation}\label{Eq:Ma:0}
\begin{cases}
S_k(D^2u)= \lambda  \frac{|x|^{\mu-2}}{(1+|x|^2)^{\frac{\mu}{2}}} (1-u)^q &\mbox{in }\;\; \Omega,\\
u <0 & \mbox{in }\;\; \Omega,\\
u=0 &\mbox{on }\partial \Omega,
\end{cases}
\end{equation}
where $\lambda>0$ is an additional parameter, $q>k$, $\mu\geq 2$, and $\Omega$ is a suitable bounded domain in $\RR^n$. We point out that answers to the above questions raised for \eqref{Eq:Ma:0} are, to our knowledge, unknown in the literature. We handle the existence and multiplicity of radially symmetric bounded solutions combining dynamical-systems tools with the approach of the intersection number between a regular solution and a singular solution. To this end, we first use a new transformation recently introduced in \cite{SaVe17}, which reduces the radial version of \eqref{Eq:Ma:0} (denoted $(P_{\lambda})$) to a two-dimensional non-autonomous Lotka-Volterra system (denoted $(MS_{q,\mu})$). This non-autonomous system can be considered as an asymptotically autonomous system in the sense of Thieme \cite{Thie94}. Thus we focus on the corresponding limiting systems, particularly in case $t\to -\infty$ (denoted $(LVS_{q,\rho_-})$), which allows us to obtain two relevant exponents for system $(MS_{q,\mu})$: They are the Tso and Joseph-Lundgren type exponents. It is worth to mention that system $(LVS_{q,\rho_-})$ matches up with an autonomous Lotka-Volterra system obtained for studying problem \eqref{Eq:Ma:0} with a power weight on the right hand side equal to $|x|^{\mu -2}$. See \cite{SaVe17} for more details.

The flow of system $(MS_{q,\mu})$ is analyzed from the corresponding limiting systems; two critical points (denoted $(\hat x, \hat y)$ and $P_3(n-2+\mu , 0)$) of $(MS_{q,\mu})$, which also are critical points of $(LVS_{q,\rho_-})$, are the key to obtaining a singular solution and a bounded solution for problem $(P_{\lambda})$. More precisely, we show that the orbits of $(MS_{q,\mu})$ starting at the critical point $P_3(n-2+\mu , 0)$ are characterized by the existence of bounded solutions to problem $(P_{\lambda})$ (see Proposition \ref{Prop:Equiv} below). On the other hand, the orbits of $(MS_{q,\mu})$ starting at the critical point $(\hat x, \hat y)$ yield with existence of a singular solution to $(P_{\lambda})$ for some $\lambda>0$ (denoted $\tilde{\lambda}$). This new parameter $\tilde{\lambda}$ is essential to obtaining the multiplicity of radially symmetric bounded solutions to \eqref{Eq:Ma:0}.

A general existence result of solutions to \eqref{Eq:Ma:0} is obtained basically by the super and subsolutions method. See Section \ref {Sec:Exist}.

\medbreak

The paper is organized as follows. In Section 2 we briefly describe the $k$-Hessian operator and introduce some basic definitions. Theorems \ref{Exist} and \ref{S4L4}, which are our main results, are established in this section. In Section 3 we prove a general existence result of classical solutions of \eqref{Eq:Ma:0} (see Lemma \ref{max:sol}) and we use this result to prove Theorem \ref{Exist}. In Section 4 we obtain a proper non-autonomous Lotka-Volterra System from which we construct a singular solution, using the contraction mapping theorem. In Section 5 we study the intersection number between a regular solution and a singular solution of suitable equations. Finally, in Section 6 we prove Theorem \ref{S4L4}.

\medbreak

\section{Preliminaries and main results}

The $k$-Hessian operator $S_k$ is defined as follows. Let $k\in\NN$ and let $\Omega$ be a suitable bounded domain in $\RR^N$. Let $u\in C^2(\Omega)$, $1\leq k\leq n$, and let $\Lambda=(\lambda_1,\lambda_2,...,\lambda_n)$ be the eigenvalues of the Hessian matrix $(D^2u)$. Then the $k$-Hessian operator is given by the formula
\[
S_k(D^2u)=P_k(\Lambda)=\sum_{1\leq i_1<...<i_k\leq n}\lambda_{i_1}...\lambda_{i_k},
\]
where $P_k(\Lambda)$ is the $k$-th elementary symmetric polynomial in the eigenvalues $\Lambda$. This operator has a long history, see e.g. \cite{CaNS85, ChWa01, Trudinger95, Tso89, Tso90, Wang94, Wang09} and the references therein. Note that they include the usual Laplace operator ($k=1$). Recently, this class of operators has attracted renewed interest, see e.g. \cite{Bran13, DeGa14, Gavitone09, Gavitone10, NaTa15, SaVe16, WaBa13, WaXu14, Wei16, Wei17}.

\medbreak

Let $\Omega=B$ be the unit ball in $\RR^n$, which is an admissible domain for $S_k$. Then the $k$-Hessian operator when acting on radially symmetric functions can be written as $S_k(D^2u)=c_{n,k}\,r^{1-n}\left(r^{n-k}(u')^k \right)'$, where $r=|x|,\,x\in\RR^n$ and $c_{n,k}$ is defined by $c_{n,k}=\binom{n}{k}/n$.

Thus we can write \eqref{Eq:Ma:0} in radial coordinates, i.e.,
\begin{equation*}
(P_{\lambda})\qquad
\begin{cases}
c_{n,k}r^{1-n}\left(r^{n-k}(u')^k \right)' = \lambda \frac{r^{\mu-2}}{(1+r^2)^{\frac{\mu}{2}}} (1-u)^q, &  0<r<1,\\
u(r)  < 0, &  0\leq r<1,\\
u'(0)  =0,\, u(1)=0. &
\end{cases}
\end{equation*}

We introduce the space of functions $\Phi_0^k$ defined on $(0,1)$ for problem $(P_{\lambda})$:
\[
\Phi_0^k=\{u\in C^2((0,1))\cap C^1([0,1]): \left(r^{n-i}(u')^i \right)'\geq 0\;\;\mbox{in}\;\;(0,1) ,\, i=1,...,k,\,u'(0)=u(1)=0\}.
\]
Note that the functions in $\Phi_0^k$ are non-positive on $[0,1]$. However, if $\left(r^{n-i}(u')^i \right)'> 0$ for every $i=1,\ldots, k$, then every function in $\Phi_0^k$ is negative and strictly increasing on $(0,1)$. This in turn implies, as we are looking for solutions of ($P_\lambda$) in $\Phi_0^k$, that the parameter $\lambda$ must be positive.

\begin{definition}
Let $\lambda>0$. We say that a function $u \in C([0,1])$ is:
\begin{itemize}
\item[(i)] a {\it classical solution} of ($P_\lambda$) if $u\in \Phi_0^k$ and the first equality in ($P_\lambda$) holds;
\item[(ii)] an {\it integral solution} of ($P_\lambda$) if $u$ is absolutely continuous on $(0,1]$, $u(1)=0$, $\int_0^1 r^{n-k}(u'(r))^{k+1} dr <\infty$ and the equality
\[
c_{n,k}r^{n-k}(u'(r))^k = \lambda\int_0^r s^{n-1} \frac{s^{\mu-2}}{(1+s^2)^{\frac{\mu}{2}}}(1-u(s))^qds,\, \, \text{ a.e. } r\in (0,1),
\]
holds whenever the integral exists.
\end{itemize}
\end{definition}

The concept of integral solution was introduced in \cite{ClFM96} for a more general class of radial operators, see e.g. \cite{ClFM96} and the references therein.  The standard concept of weak solution is equivalent in this case to the notion of integral solution, see \cite[Proposition 2.1]{ClFM96}.

\medbreak

We recall the version of the method of super and subsolutions for \eqref{Eq:Ma:0}, see \cite[Theorem 3.3]{Wang94} for more details.
\begin{definition}
A function $u\in\Phi^k(B):=\{u\in C^2(B)\cap C(\overline{B}): S_{i}(D^2 u)\geq 0\;\;\mbox{in}\;B,\, i=1,...,k\}$ is called a {\it subsolution} (resp. {\it supersolution}) of \eqref{Eq:Ma:0} if
\begin{equation*}
\begin{cases}
S_k(D^2u)\geq (\mbox{resp.}\leq)& \lambda  \frac{|x|^{\mu-2}}{(1+|x|^2)^{\frac{\mu}{2}}} (1-u)^q \;\;\mbox{in }\;\; B,\\
u\leq (\mbox{resp.}\geq)\;\; 0&\qquad\qquad\;\,\mbox{on }\; \partial B.
\end{cases}
\end{equation*}
\end{definition}
Note that the trivial function $u\equiv 0$ is always a supersolution.

The following concept is needed to establish a general result on the existence of solutions to problem \eqref{Eq:Ma:0}.
\begin{definition}
We say that a function $v$ is a {\it maximal} solution of \eqref{Eq:Ma:0} if $v$ is a solution of \eqref{Eq:Ma:0} and, for each subsolution $u$ of \eqref{Eq:Ma:0}, we have $u\leq v$.
\end{definition}
This notion of maximal solution was recently introduced in \cite{SaVe16} to prove existence results, see also \cite{SaVe17}.
\medbreak

Now we state our first main result concerning the existence and non-existence of solutions to problem $(P_\lambda)$.
\begin{theorem}\label{Exist}
Let $n > 2k$, $q > k$ and $\mu\geq 2$. There exists $\lambda^*>0$ such that problem $(P_\lambda)$ admits a maximal bounded solution for $\lambda\in (0,\lambda^*)$, at least one possibly unbounded integral solution for $\lambda=\lambda^*$ and no classical solutions for all $\lambda>\lambda^*$. Additionally, 
\[
\lambda^* \geq d(\mu)\binom{n}{k}\left(\frac{2k}{q-k}\right)^k\left(\frac{q-k}{q}\right)^q,
\]
where the positive constant $d(\mu)$ is given by
\[
d(\mu) =
\begin{cases}
1\;\;& \mbox{if }\;\;\mu=2,\\
\frac{(\frac{\mu}{2})^\frac{\mu}{2}}{(\frac{\mu-2}{2})^\frac{\mu-2}{2}}\;\;& \mbox{if }\;\;2< \mu\leq 4,\\
2^{\frac{\mu}{2}}\;\;& \mbox{if }\;\;\mu>4.
\end{cases}
\]
\end{theorem}

Next, in order to state our second main result we introduce two relevant exponents. Let $\sigma\geq 0$. From now on we shall denote by  
\[
q^*(k,\sigma)=\frac{(n+2)k+\sigma(k+1)}{n-2k}
\] 
and
\begin{equation*}\label{Exp:critical:Intro}
q_{JL}(k,\sigma):=
\begin{cases}
k\frac{k(k+1)n-k^2(2-\sigma)+2k+\sigma-2\sqrt{k(2k+\sigma)[(k+1)n-k(2-\sigma)]}}{k(k+1)n-2k^2(k+3)-2k\sigma-2\sqrt{k(2k+\sigma)[(k+1)n-k(2-\sigma)]}}, & n>2k+8+\frac{4\sigma}{k},\\
\infty, & 2k < n \leq 2k+8+\frac{4\sigma}{k},
\end{cases}
\end{equation*}

\noindent the Tso and Joseph-Lundgren type exponents, respectively.

The generalized Joseph-Lundgren exponent, $q_{JL}(k,\sigma)$, was recently obtained in \cite{SaVe17} in connection with the multiplicity of radial bounded solutions of a $k$-Hessian equation involving a weight of the form $|x|^{\sigma}$. We point out that, for $k = 1$ and $\sigma=0$, $q_{JL}(1,0)$ coincides with the classical Joseph-Lundgren exponent \cite{JoLu73}. When $k>1$ and $\sigma=0$, this exponent also appears in a large class of problems with nonlinear radial operators including the usual Laplace, $p$-Laplace and $k$-Hessian operators \cite{Miya16, MiTa17}. See also \cite{SaVe16} for the case of the $k$-Hessian operator.
   
Now we state our second main result.
\begin{theorem}\label{S4L4}
Let $n > 2k$, $q > k$ and $\mu\geq 2$. Assume that $q^*(k,\mu-2)< q < q_{JL}(k,\mu-2)$. Then there exists a positive constant $\tilde{\lambda} < \lambda ^*$ such that for each $N\ge 1$, there is an $\e>0$ such that if $|\lambda-\tilde{\lambda}|<\e$, then $(P_{\lambda})$ has at least $N$ solutions.
In particular, if $\lambda=\tilde{\lambda}$, then $(P_{\lambda})$ has infinitely many solutions.
\end{theorem}

It is remarkable that the exponents $q^*(k,\mu-2)$ and $q_{JL}(k,\mu-2)$ have the same role for problem $(P_{\lambda})$ with different types of weights on the right hand side, either $\frac{r^{\mu-2}}{(1+r^2)^\frac{\mu}{2}}$ or $r^{\mu-2}$, see \cite[Theorem 3.1 (I)]{SaVe17} for the last case. On the other hand, even though the corresponding parameter $\tilde{\lambda}$ also plays the same role in both cases, this value of the parameter $\lambda$ is different in each case. See Lemma \ref{S3L2} below for the definition of $\tilde{\lambda}$, and see \cite[Theorem 3.1 (I)]{SaVe17} for its definition in case of $r^{\mu-2}$. In Section \ref{InterNumber} we discuss further these relationships. 

\section{Existence and non-existence of solutions of problem $(P_\lambda)$}\label{Sec:Exist}
In this section we prove a general existence result of classical solutions of problem $(P_\lambda)$.
We begin with 
\begin{lemma}\label{max:sol}
Let $n > 2k$, $q > k$, $\mu\geq 2$ and $\lambda_0 >0$. Assume that there exists a classical solution of
\begin{equation}
\begin{cases}\label{Eq:2:0}
c_{n,k}r^{1-n}\left(r^{n-k}(w')^k \right)' = \lambda_0 \frac{
r^{\mu-2}}{(1+r^2)^\frac{\mu}{2}} (1-w)^q\,,\quad 0<r<1,\\
w  < 0 \,, \hspace{4.65cm} 0\leq r<1,\\
w'(0)  =0,\, w(1)=0. &
\end{cases}
\end{equation}

Then, for every $\lambda\in (0,\lambda_0)$, problem $(P_\lambda)$ has a classical maximal bounded solution. Moreover, the classical maximal bounded solutions form a decreasing sequence as $\lambda$ increases.
\end{lemma}
\begin{proof}
Fix $\lambda\in (0,\lambda_0)$ and define the functions
\[
g(t) = \left[\lambda_0(1+t)^q\right]^{1/k}\, \text{ and }\;\, \tilde{g}(t) = \left[\lambda(1+t)^q\right]^{1/k},\text{ for all }\; t\geq 0.
\]
Set $\Phi(s) = \tilde{h}^{-1}(h(s))$ ($s\leq 0$) with $h$ and $\tilde{h}$ given by
\[
h(s) = \int_s^0 \frac{1}{g(-t)}\,dt\; \text{ and }\;\, \tilde{h}(s) = \int_s^0 \frac{1}{\tilde{g}(-t)}\,dt, \; s\leq 0.
\]
Since $q>k$, $\lim_{s\to -\infty}h(s)$ exists and hence $\Phi$ is bounded by \cite[Lemma 2.1 (i)-(ii)]{SaVe16}. Next, by \eqref{Eq:2:0} and the convexity of $\Phi$ \cite[Lemma 2.1 (iii)]{SaVe16}, we have
\begin{eqnarray*}
S_k(D^2 \Phi(w))&=&c_{n,k}kr^{1-k}(\Phi'(w)w')^{k-1}\left(\Phi''(w)(w')^2+\Phi'(w)w''+\frac{n-k}{k}\frac{\Phi'(w)w'}{r}\right)\\
&\geq & c_{n,k}kr^{1-k}(\Phi'(w))^k(w')^{k-1}\left(w''+\frac{n-k}{k}\frac{w'}{r}\right)\\
& = &(\Phi'(w))^k S_k(D^2 w) = \frac{(\tilde{g}(-\Phi(w)))^k}{(g(-w))^k}S_k(D^2 w) = \lambda \frac{
r^{\mu-2}}{(1+r^2)^\frac{\mu}{2}}(1-\Phi(w))^q.
\end{eqnarray*}
Therefore $\Phi(w)$ is a bounded subsolution of $(P_\lambda)$ and thus, by the method of super and subsolutions, we have, by \cite[Theorem 3.3]{Wang94}, a solution $u\in L^{\infty}((0,1))$ of $(P_\lambda)$ with $\Phi(w) \leq u \leq 0$. Now, to prove that ($P_{\lambda}$) admits a maximal solution, we consider $u_1$ as the solution of
\begin{equation*}
(Q)\;\;\begin{cases}
S_k(D^2 u_1)= \lambda \frac{
|x|^{\mu-2}}{(1+|x|^2)^\frac{\mu}{2}}&\mbox{in }\;\; B,\\
u_1=0 &\mbox{on }\; \partial B.
\end{cases}
\end{equation*}
Note that $u_1\in \Phi^k_0(B)$ since $\mu\geq 2$. As $u$ is in particular a subsolution of $(Q)$, we have $u\leq u_1$ on $B$ by the comparison principle \cite{TrXu99}. Next, we define $u_i$ ($i=2,3,\ldots$) as the solution of
\begin{equation*}
\begin{cases}
S_k(D^2 u_i)= \lambda \frac{
|x|^{\mu-2}}{(1+|x|^2)^\frac{\mu}{2}}(1-u_{i-1})^q &\mbox{in }\;\; B,\\
u_i=0 &\mbox{on }\; \partial B.
\end{cases}
\end{equation*}
Note that $u_i\in \Phi^k_0(B)$ ($i=2,3,\ldots$) since $\mu\geq 2$. Using again the comparison principle we obtain a increasing sequence $\{u_i\}$, which is bounded from below by $u$ and by 0 from above. Hence, we can pass to the limit to obtain a classical solution $u_{max}$ of ($P_{\lambda}$), which is maximal since the recursive sequence $\{u_i\}$ does not depend on the subsolution $u$. Now let $\lambda_1<\lambda_2$ and $u_{\lambda_1}$, $u_{\lambda_2}$ be maximal solutions of $(P_{\lambda_i})$ ($i=1,2$), respectively. Since $u_{\lambda_2}$ is a subsolution of $(P_{\lambda_1})$, we have $u_{\lambda_2} \leq u_{\lambda_1}$ by the maximality of $u_{\lambda_1}$.
\end{proof}

\subsection{Proof of Theorem \ref{Exist}}

Fix $\mu\geq 2$. For $R>1$, let $B_R$ be a ball centered at zero with radius $R$ such that $\overline{B}\subset B_R$, and let $\eta$ be the solution of
\begin{equation*}
\begin{cases}
S_k(D^2\eta)= 1 &\mbox{in }\;\; B_R,\\
\eta=0 &\mbox{on }\; \partial B_R.
\end{cases}
\end{equation*}
Then there exists a constant $\beta$ such that $\eta<\beta<0$ on $\partial{B}$.
Set $M=\max_{x\in\overline{B}}\,|\eta(x)|$, $C=C(\mu)=\max_{r\in [0,1]} \frac{
r^{\mu-2}}{(1+r^2)^\frac{\mu}{2}}>0$ and take $\lambda<C^{-1}(1+M)^{-q}$. Then
\[
S_k(D^2\eta)=1>\lambda C(1+M)^q\geq\lambda C(1-\eta)^q\geq\lambda \frac{
|x|^{\mu-2}}{(1+|x|^2)^\frac{\mu}{2}}(1-\eta)^q \;\;\mbox{in}\;\;B.
\]
By \cite[Theorem 3.3]{Wang94}, for every $\lambda\in (0,C^{-1}(1+M)^{-q})$ there exists a solution $u_{\lambda}$ of $(P_\lambda)$. Thus we may define
\begin{equation}\label{lambda-ast}
\lambda^\ast=\sup\{\lambda>0:\, \text{ there is a solution } u_{\lambda}\in C^2(B) \text{ of } \eqref{Eq:Ma:0}\}.
\end{equation}
Then $\lambda^\ast>0$.

To see that $\lambda^\ast$ is finite, we consider the inequality
\begin{equation}\label{Eq:Nineq}
\Delta u\geq C(n,k)[S_k(D^2 u)]^\frac{1}{k},
\end{equation}
which holds for every $u\in \Phi^k(B)$, see e.g. \cite[Proposition 2.2, part (4)]{Wang94} and the comments therein. Consider now the eigenvalue problem
\begin{equation*}
(E_m)\;\;
\begin{cases}
-\Delta u= \lambda m(x)u &\mbox{in }\;\; B,\\
u=0 &\mbox{on }\; \partial B,
\end{cases}
\end{equation*}
where $m(x):=h(|x|)^\frac{1}{k}$. It is known that problem $(E_m)$ has a first eigenvalue, $\lambda_{1,m}>0$, associated with an eigenfunction $\phi_{1,m}>0$, see e.g. \cite[Theorem 0.6]{AmPr93}. It is not difficult to see that there exists a constant $M>0$ such that for every $u<0$, we have $(1-u)^\frac{q}{k}\geq \frac{M|u|}{C(n,k)}$. Let $\lambda\in (0,\lambda^\ast)$ and let $u$ be a solution of problem $(P_\lambda)$. Then, using \eqref{Eq:Nineq}, we obtain
\begin{eqnarray*}
\lambda_{1,m}\int_B |u|m(x)\phi_{1,m}&=&-\int_B\Delta\phi_{1,m}|u|\\
&=&-\int_B Du D\phi_{1,m}\\
&\geq & M\lambda^\frac{1}{k}\int_B |u|m(x)\phi_{1,m},
\end{eqnarray*}
which in turn implies that $\lambda\leq\left(\frac{\lambda_{1,m}}{M}\right)^k$. Thus $\lambda^\ast$ is finite.

Now, let $\lambda\in (0,\lambda^{\ast})$. Then $u_\lambda$ is a maximal bounded solution of ($P_{\lambda}$) by Lemma \ref{max:sol} applied to $\lambda_0 \in (\lambda, \lambda^{\ast})$.

\medbreak

Next let $\lambda_i$ be an increasing sequence such that $\lambda_i\rightarrow \lambda^\ast$ as $i\rightarrow +\infty$ and let $u_{\lambda_i}$ be a maximal solution of $(P_{\lambda_i})$. By Lemma \ref{max:sol}, for all $r\in [0,1]$ we have $u_{\lambda_{i+1}}(r)\leq u_{\lambda_{i}}(r)\leq 0$. On the other hand, integrating the first equation in $(P_{\lambda_i})$, we obtain
\[
u_{\lambda_i}(r)=-\int_{r}^{1}\left[c_{n,k}^{-1}\,\tau^{k-n}\int_{0}^{\tau}s^{n-1}\frac{s^{\mu-2}}{(1+s^2)^\frac{\mu}{2}}\lambda_{i}(1-u_{\lambda_i}(s))^q\,ds\right]^{\frac{1}{k}}d\tau.
\]
Applying the monotone convergence theorem twice, we conclude that
\[
u^\ast(r):=\lim_{i\rightarrow +\infty}u_{\lambda_i}(r), \;\;\; \text{ exists a.e. } r\in  (0,1)
\]
and
\[
u^\ast(r)=-\int_{r}^{1}\left[c_{n,k}^{-1}\,\tau^{k-n}\int_{0}^{\tau}s^{n-1}\frac{s^{\mu-2}}{(1+s^2)^\frac{\mu}{2}}\lambda^\ast(1-u^\ast(s))^q\,ds\right]^{\frac{1}{k}}d\tau, \;\;\; \text{ a.e. } r\in  (0,1).
\]
The assertion concerning the non-existence of solutions follows directly from the definition of $\lambda^\ast$. 

Next we obtain a lower bound for $\lambda^*$ in Theorem \ref{Exist}. The function $v(x)=\frac{k}{q-k}\left(|x|^2-1\right)$ satisfies
\[
S_k(D^2 v)=nc_{n,k}(2k)^k(q-k)^{-k}\geq \binom{n}{k}(2k)^k\frac{(q-k)^{q-k}}{q^q}(1-v)^q\geq C^{-1}\binom{n}{k}(2k)^k\frac{(q-k)^{q-k}}{q^q}h(r)(1-v)^q,
\]
where the constant $C=C(\mu)$ is given by
\[
C=\max_{r\in [0,1]}h(r)=\max_{r\in [0,1]}\frac{r^{\mu-2}}{(1+r^2)^\frac{\mu}{2}}=
\begin{cases}
1\;\;& \mbox{if }\;\;\mu=2,\\
\frac{(\frac{\mu-2}{2})^\frac{\mu-2}{2}}{(\frac{\mu}{2})^\frac{\mu}{2}}\;\;& \mbox{if }\;\;2< \mu\leq 4,\\
2^{-\frac{\mu}{2}}\;\;& \mbox{if }\;\;\mu>4.
\end{cases}
\]
Hence $v$ is a subsolution of $(P_\lambda)$ for all $\lambda\leq C^{-1}\binom{n}{k}(2k)^k\frac{(q-k)^{q-k}}{q^q}$. Since $v_0\equiv 0$ is a supersolution and $v\leq v_0$, for any such $\lambda$ there exists a solution of $(P_\lambda)$. By the first statement of Theorem \ref{Exist}, this shows that
\[
\lambda^*\geq C^{-1}\binom{n}{k}\left(\frac{2k}{q-k}\right)^k\left(\frac{q-k}{q}\right)^q.
\]
Setting $d(\mu) = C^{-1}$ we conclude the proof of Theorem \ref{Exist}.\hfill $\square$

\section{Existence of a singular solution of $(P_{\lambda})$}\label{SingularSection}

In this section we obtain a singular solution of $(P_{\lambda})$ which is derived from a proper non-autonomous Lotka-Volterra system.

\subsection{A non-autonomous Lotka-Volterra system}

We start considering the radial problem
\begin{equation}\label{RaPr:1}
\begin{cases}
\left(r^{n-k}(u')^k\right)'= c_{n,k}^{-1}\,r^{n-1}f(r,u),& 0<r<1,\\
u(r)<0, &0\leq r<1,\\
u'(0) = 0,\, u(1) =0.
\end{cases}
\end{equation}
Let $u$ be a solution of \eqref{RaPr:1} and set $w=u-1$. Then $w$ is a solution of
\begin{equation}\label{Eq:IVP:0}
\left(r^{n-k}(w')^k\right)'= r^{n-1}c_{n,k}^{-1}\,f(r,w+1), \;\; (r>0).
\end{equation}
To obtain a Lotka-Volterra system, we set
\begin{equation}\label{newtrans0}
x(t)=r^k\frac{c_{n,k}^{-1}\,f(r,w+1)}{(w')^{k}},\; y(t)=r\frac{w'}{-w}, \;r=e^t,
\end{equation}
where $w'=dw/dr$. We point out that this change of variable is well-known in the case $k=1$, see e.g. \cite{BatLi10, BaPf88, WaZhLi12, Wola99}. In the framework of the $k$-Hessian operator, this transformation has been recently introduced in \cite{SaVe17}. Further, for
\begin{equation}\label{ache}
f(r,w+1)=\lambda h(r)(-w)^q
\end{equation}
we see that such a change of variables becomes optimal for equation \eqref{Eq:IVP:0} since, depending on the weight $h$, we obtain either an autonomous or a non-autonomous Lotka-Volterra system. More precisely, after some calculation one can see that the pair of functions $(x(t), y(t))$ solves the following non-autonomous Lotka-Volterra system:
\begin{equation*}
(LVS_{q,\rho}) \qquad\begin{cases}
\frac{dx}{dt}=x\left[\rho(t)-x-q y\right],
\,\\
\frac{dy}{dt} = y\left[-\frac{n-2k}{k}+\frac{x}{k}+y\right],
\end{cases}
\end{equation*}
where $\rho(t)=n+r\frac{h'(r)}{h(r)}$\; and\; $r=e^t$.

Note that we can recover the function $w$ by the formula
\begin{equation}\label{inverse0}
w=-\left[c_{n,k}^{-1}\lambda\,r^{2k}h(r)\right]^{-\frac{1}{q-k}}(xy^k)^{\frac{1}{q-k}}.
\end{equation}

Now, in order to transform the problem $(P_\lambda)$ into the system $(LVS_{q,\rho})$, we set
\[
h(r)=\frac{r^{\mu-2}}{(1+r^2)^{\frac{\mu}{2}}}
\]
thus obtaining the non-autonomous dynamical system:
\begin{equation*}
(MS_{q,\mu})\qquad\begin{cases}
\frac{dx}{dt}=x\left[n-2  + \frac{\mu}{1+e^{2t}}-x-q y\right],
\,\\
\frac{dy}{dt} = y\left[-\frac{n-2k}{k}+\frac{x}{k}+y\right].
\end{cases}
\end{equation*}

Since the limits $\lim_{t\to \pm \infty} \rho(t) =\lim_{t\to \pm \infty}( n-2  + \frac{\mu}{1+e^{2t}})=:\rho_{\pm}$ exist, we may consider the system $(MS_{q,\mu})$ as an asymptotically autonomous system in the sense of Thieme (see \cite{Thie94}). Thus we can describe the flow of $(MS_{q,\mu})$ from the autonomous systems $(LVS_{q,\rho_{\pm}})$.

\subsection{The linearization of $(LVS_{q,\rho_{\pm}})$ at the stationary points}\label{LinSPoints}

The following decompositions:
\begin{eqnarray}
\rho(t) & = n-2+\mu -\frac{\mu e^{2t}}{1+e^{2t}}\label{rho:minus}\\
& = n-2 +  \frac{\mu e^{-2t}}{1+ e^{-2t}},\label{rho:plus}
\end{eqnarray}
are useful when $t\to -\infty$ and $t\to +\infty$, respectively.

Now let $P=(a,b)$ be a stationary point of $(LVS_{q,\rho_{\pm}})$, if we introduce the coordinates $\bar{x}:=x-a$ and $\bar{y}:=y-b$, then using e.g. \eqref{rho:minus} we can write $(MS_{q,\mu})$ as a time-dependent perturbation of $(LVS_{q,\rho_{-}})$:
\begin{equation}\label{perturba}
\begin{pmatrix}
\frac{d \bar{x}}{dt}\\
\\
\frac{d \bar{y}}{dt}
\end{pmatrix}
= A \begin{pmatrix}
 \bar{x}\\
\\
\bar{y}
\end{pmatrix}
+ \begin{pmatrix}
-\bar{x}^2 -q\bar{x}\bar{y}\\
\\
\frac{\bar{x}\bar{y}}{k} + \bar{y}^2
\end{pmatrix}
+ \begin{pmatrix}
-(\bar{x}+a)\frac{\mu e^{2t}}{1+e^{2t}}\\
\\
0
\end{pmatrix}
\end{equation}
with
\[
A:=
\begin{pmatrix}
n-2+\mu -2a -qb & -qa\\
\\
\frac{b}{k} & \frac{a}{k} + 2b - \frac{n-2k}{k}
\end{pmatrix}
\]

{\bf Case $(LVS_{q,\rho_-})$}: The critical points of $(LVS_{q,\rho_-})$ are $P_1(0,0), P_2(0,\frac{n-2k}{k}), P_3(n-2+\mu,0)$, and
\begin{equation}\label{interiorcritipoint:minus}
\left(\frac{q(n-2k)- k(n-2+\mu)}{q-k}, \frac{2k-2+\mu}{q-k}\right)=:(\hat{x},\hat{y}).
\end{equation}
Note that, under the assumptions $2k<n,\, k<q$ and $2\leq\mu$, the first three critical points belong to $\RR^2_+:=\{(x,y)\in\RR^2:x\geq 0,\,y\geq 0\}$ and they are saddle points. The fourth critical point $(\hat{x},\hat{y})$ belongs to the interior of $\RR^2_+$ if, and only if, $q> k(n-2+\mu)/(n-2k)$. Further, $(\hat{x},\hat{y})$ is a stable node for $q> q^*(k,\mu-2)=\frac{(n+2)k+(\mu-2)(k+1)}{n-2k}$. It is not difficult to see that the (bounded) orbit $(x(t),y(t))$ of $(LVS_{q,\rho_-})$ starts at $P_3(n-2+\mu,0)$. See \cite{SaVe17}.

\medbreak

{\bf Case $(LVS_{q,\rho_+}) $}: Considering the decomposition of $\rho(t)$ as in \eqref{rho:plus}, we obtain
\[
\begin{pmatrix}
\frac{d \bar{x}}{dt}\\
\\
\frac{d \bar{y}}{dt}
\end{pmatrix}
= A \begin{pmatrix}
 \bar{x}\\
\\
\bar{y}
\end{pmatrix}
+ \begin{pmatrix}
-\bar{x}^2 -q\bar{x}\bar{y}\\
\\
\frac{\bar{x}\bar{y}}{k} + \bar{y}^2
\end{pmatrix}
\]
with
\[
A:=
\begin{pmatrix}
n-2 -2a -qb & -qa\\
\\
\frac{b}{k} & \frac{a}{k} + 2b - \frac{n-2k}{k}.
\end{pmatrix}
\]
The critical points of $(LVS_{q,\rho_+})$ are $P_1(0,0), P_2(0,\frac{n-2k}{k}), P_3(n-2,0)$, and
\begin{equation}\label{interiorcritipoint:plus}
\left(\frac{q(n-2k)- k(n-2)}{q-k}, \frac{2k-2}{q-k}\right)=:(\tilde{x},\tilde{y}).
\end{equation}
We point out that $(LVS_{q,\rho_+})$ has four critical points whenever that $k>1$. In the case $k=1$ the critical point $(\tilde{x},\tilde{y})$ coincides with $P_3(n-2,0)$. Further, the first three critical points are saddle points and $(\tilde{x},\tilde{y})$ is a stable focus provided that $q>q^*(k,-2)$  where
\begin{equation}\label{q:ast:kalfa}
q^*(k,-2) :=\frac{nk-2}{n-2k}.
\end{equation}
Note that $q^*(k,-2) > k(n-2)/(n-2k)$. Further, for all $q\geq k(n-2)/(n-2k)$, $(\tilde{x}, \tilde{y})$ belongs to $\RR^2_+$ and for $q\geq q^*(k,\mu-2)$ both critical points $(\tilde{x}, \tilde{y})$ and $(\hat x, \hat y)$ belong to the interior of $\RR^2_+$.

\subsection{The flow of $(MS_{q,\mu})$}
Note that $\rho(t)=n-2+\mu-\frac{\mu e^{2t}}{1+e^{2t}}$ is strictly decreasing from $\rho(-\infty)=n-2+\mu$ to $\rho(+\infty)=n-2$ on $\bar{\RR}:=[-\infty,+\infty]$. We define
\begin{eqnarray*}
S(t,x,y)&:=&\rho(t)-x-q y,\;\; t\in\bar{\RR},\\
W(x,y)&:=&-\frac{n-2k}{k}+\frac{x}{k}+y
\end{eqnarray*}
and we write $(MS_{q,\mu})$ in the form
\begin{eqnarray*}
x'&=& x S(t,x,y)\\
y'&=& y W(x,y).
\end{eqnarray*}
For functions $F:\RR^+\times\RR^+\rightarrow\RR$, we let
\[
F_0:=F^{-1}\{0\},\;\; F_+:=F^{-1}\{\RR^+\},\;\; F_-:=F^{-1}\{\RR^-\}.
\]
Now $S_0(t),\; t\in\bar{\RR}$, is the straight line $y=-\frac{1}{q}\,x+\frac{\rho(t)}{q}$ with intercepts $(\rho(t),0)$ and $\left(0,\frac{\rho(t)}{q}\right)$. It moves downward in the closed strip $Z$ between $S_0(-\infty)$ and $S_0(+\infty)$ when $t$ runs from $-\infty$ to $+\infty$. On the other hand, $W_0$ is the fixed straight line $y=\frac{n-2k}{k}-\frac{x}{k}$. Furthermore, we define
\[
G(x,y):=x+\frac{(n-2k)(q+1)}{k+1}\left(\frac{k}{n-2k}\,y-1\right).
\]
Note that $G_0$ is the straight line $y=-\frac{k+1}{k(q+1)}\,x+\frac{n-2k}{k}$ with intercepts $\left(\frac{(n-2k)(q+1)}{k+1},0\right)$ and $(0,\frac{n-2k}{k})$.
\begin{lemma}\label{inward}
Let $\varphi$ be a solution of $(MS_{q,\mu})$ on $(T_0,T)$. Let $q\geq q^*(k,\mu-2)$.
\begin{itemize}
\item [i)] If $\varphi(t_0)\in G_{-}\cup G_0\;\;\mbox{for}\;\;t_0\geq -\infty$, then $\varphi(t)\in G_-$ for $t>t_0$.
\item [ii)] If $\varphi(t_0)\in G_{+}\cup G_0\;\;\mbox{for}\;\;t_0>-\infty,$ then $\varphi(t)\in G_+$ for $-\infty\leq t<t_0$.
\end{itemize}
\begin{proof}
We have
\begin{eqnarray*}
\frac{d}{dt}\,G(\varphi(t))&=&x'+\frac{k(q+1)}{k+1}\,y'\\
&=& x(\rho(t)-x-q y)+\frac{k(q+1)}{k+1}\,y\left(-\frac{n-2k}{k}+\frac{x}{k}+y\right).
\end{eqnarray*}
If $\varphi(t)\in G_0$, then $-x=\frac{(n-2k)(q+1)}{k+1}\left(\frac{k}{n-2k}\,y-1\right)$, whence
\begin{eqnarray*}
\frac{d}{dt}\,G(\varphi(t))|_{\varphi(t)\in G_0}&=& x\left[\rho(t)+\frac{k(q+1)}{k+1}\,y-\frac{(n-2k)(q+1)}{k+1}-(q+1)y+\frac{q+1}{k+1}\,y\right]\\
&=&x\left[\rho(t)-\frac{(n-2k)(q+1)}{k+1}\right]<0\;\;\mbox{for all}\;\;t.
\end{eqnarray*}
\end{proof}
\end{lemma}

Next we link the solution of the initial value problem
\begin{equation}\label{IVPw}
\begin{cases}
(r^{n-k}(w')^k)'=r^{n-1}c_{n,k}^{-1}\lambda\frac{r^{\mu-2}}{(1+r^2)^{\mu/2}}(-w)^q, & r>0,\\
w(0)=w_0\in (-\infty,0),\\
w'(0)=0
\end{cases}
\end{equation}
with a property of the orbits of system $(MS_{q,\mu})$.
\begin{proposition}\label{Prop:Equiv} Suppose that $q\geq q^*(k,\mu-2)$. Then there exists a unique global solution $w$ of \eqref{IVPw} in the regularity class $C^2(0,\infty)\cap C^1[0,\infty)$. Furthermore, the function $w$ defined by \eqref{inverse0} is the unique solution of \eqref{IVPw} if, and only if, the orbit $(x(t),y(t))$ of system $(MS_{q,\mu})$ given by \eqref{newtrans0} starts at the point $P_3(n-2+\mu,0)$.
\end{proposition}
\begin{proof}
Let $w_0$ be an arbitrary negative number. Defining $B(r)=\int_0^r s^{\frac{(\alpha-\beta)(q+1)}{\beta+1}}(a(s)s^{\theta})'ds$, $r>0$, with $\alpha=n-k$, $\beta=k$, $\gamma=n-1$, $a(s)=c_{n,k}^{-1}\lambda\frac{r^{\mu-2}}{(1+r^2)^\frac{\mu}{2}}$ and $\theta=[(\gamma+1)(\beta+1)-(\alpha-\beta)(q+1)]/(\beta +1)$, we obtain $B(r)\leq 0$ for $r>0$ if, and only if, $q\geq q^*(k,\mu-2)$. Then the global existence of \eqref{IVPw} follows from \cite[Theorem 4.1]{ClMM98}. The uniqueness follows by a simple application of the contraction mapping principle, as in \cite{SaVe16}.

Next, let $w(r)$ be the solution of \eqref{IVPw}. By \eqref{newtrans0} and \eqref{IVPw}, the function $y = y(t)$ satisfies
\begin{equation}\label{limi:y}
\lim_{t\to -\infty}y(t)=\lim_{r\to 0}r\,\frac{w'(r)}{-w(r)} = 0
\end{equation}
and for $x=x(t)$, we have
$
\lim_{t\to -\infty}x(t)=\lim_{r\to 0}\lambda c_{n,k}^{-1}[-w(r)]^q\frac{r^{k}h(r)}{[w'(r)]^k},
$
with $h(r)=\frac{r^{\mu-2}}{(1+r^2)^{\frac{\mu}{2}}}$. Now, by the equation in $\eqref{IVPw}$ and L'H\^{o}\-pital's rule, we have
\begin{eqnarray*}
\lim_{r\to 0}\frac{r^{k}h(r)}{[w'(r)]^k}&=&\lim_{r\to 0}\frac{r^{n}h(r)}{r^{n-k}[w'(r)]^k}=\lim_{r\to 0}\frac{r^n h'(r)+nr^{n-1}h(r)}{r^{n-1}c_{n,k}^{-1}\lambda h(r)[-w(r)]^q}\\
&=&\lim_{r\to 0}\frac{r\frac{h'(r)}{h(r)}+n}{c_{n,k}^{-1}\lambda [-w(r)]^q}=\lim_{r\to 0}\frac{\rho(\ln r)}{c_{n,k}^{-1}\lambda [-w(r)]^q}=\frac{n-2+\mu}{c_{n,k}^{-1}\lambda (-w_0)^q}.
\end{eqnarray*}
Thus
\begin{equation}\label{limi:x}
\lim_{t\to -\infty}x(t)=n-2+\mu.
\end{equation}
From \eqref{limi:y} and \eqref{limi:x} we conclude that
\begin{equation}\label{start}
\lim_{t\to -\infty}(x(t),y(t))=(n-2+\mu,0)=P_3(n-2+\mu,0).
\end{equation}
Conversely, suppose that $(x(t),y(t))\rightarrow P_3(n-2+\mu,0)$ as $t\rightarrow -\infty$. Restricting the general linearization \eqref{perturba} to $\bar{x}=x-\rho_{-}$ with $\rho_-=n-2+\mu$ and $\bar{y}=y$, we obtain
\begin{equation}\label{xybarsystem}
\begin{cases}
\bar{x}'=-\rho_{-}\bar{x}-q\rho_{-}y-\bar{x}^2-q\bar{x}y-\mu(\bar{x}+\rho_{-})g(t),\\
y'=y\left(\hat{\mu}+\frac{\bar{x}}{k}+y\right),
\end{cases}
\end{equation}
where $g(t):=\frac{e^{2t}}{1+e^{2t}}$ and $\hat{\mu}:=\frac{2k+\mu-2}{k}$. Define $\varepsilon(t):=\frac{\bar{x}(t)}{k}+y(t)$. Then, for some $t_0$
\[
y(t)=y(t_0)e^{-\hat{\mu}t_0}e^{\hat{\mu}t+\int_{t_0}^{t}\varepsilon(s)ds}.
\]
Since $\varepsilon(t)\rightarrow 0$ as $t\rightarrow -\infty$, there exists $\delta>0$ small enough such that $|\varepsilon(t)|\leq\delta<\hat{\mu}$ for $t\leq t_0$. Thus
\[
y(t)\leq C_0 e^{(\hat{\mu}-\delta)t},\;\;\; C_0:=y(t_0)e^{-(\hat{\mu}-\delta)t_0}.
\]
If $\varepsilon$ were integrable at $-\infty$, then for some $C>0$,
\begin{equation}\label{yC}
y(t)=C e^{\hat{\mu}t}\cdot e^{\int_{-\infty}^{t}\varepsilon(s)ds},\;\; t\leq t_0.
\end{equation}
Since $\bar{x}(t)\rightarrow 0$, we may assume $\frac{\rho_-}{2}+\bar{x}\geq 0$. By \eqref{xybarsystem}, we have
\begin{eqnarray*}
\frac{1}{2}(\bar{x}^2)'&=&-(\rho_-+\bar{x})\bar{x}^2-q\bar{x}^2y-\mu\bar{x}^2 g(t)-\rho_-(qy+\mu g(t))\bar{x}\\
&\leq &-(\rho_-+\bar{x})\bar{x}^2+\frac{\rho_-}{2}[(qy+\mu g(t))^2+\bar{x}^2]\\
&\leq &\rho_-(q^2y^2+\mu^2 g^2(t)).
\end{eqnarray*}
Then
\begin{eqnarray}\label{xbar2}
\bar{x}^2(t)\leq 2\rho_-\left(q^2\int_{-\infty}^{t}y^2(s)ds+\mu^2\int_{-\infty}^{t}g^2(s)ds\right)
\end{eqnarray}
\begin{eqnarray*}
\leq 2\rho_-\left(\frac{q^2 C_{0}^2}{2(\hat{\mu}-\delta)}\;e^{2(\hat{\mu}-\delta)t}+\mu^2 e^{4t}\right).
\end{eqnarray*}
Thus $\bar{x}(t)=O(e^{\min\{2,\hat{\mu}-\delta\}t})$. We have shown that $\varepsilon$ is integrable at $-\infty$. Hence using $\eqref{yC}$
\begin{eqnarray*}
y(t)&=&C e^{\hat{\mu}t}\left[1+O\left(\int_{-\infty}^{t}\left(\frac{\bar{x}(s)}{k}+y(s)\right)ds\right)\right]\\
&=&C e^{\hat{\mu}t}\left[1+O(e^{2t})\right].
\end{eqnarray*}
Together with \eqref{xbar2} we get $\bar{x}(t)=O(e^{2t})$, and using \eqref{xybarsystem}
\begin{eqnarray*}
\bar{x}'&=& -(n-2+\mu)\bar{x}-\mu (n-2+\mu)e^{2t}-q(n-2+\mu)C e^{\hat{\mu}t}+O(e^{4t})\\
\left[e^{(n-2+\mu)t}\bar{x}\right]'&=&-\mu(n-2+\mu)e^{(\mu+n)t}-q(n-2+\mu)C e^{\frac{(k+1)\mu+nk-2}{k}\,t}+O\left(e^{(n+2+\mu)t}\right)\\
\bar{x}(t)&=& -\frac{\mu (n-2+\mu)}{\mu +n}\,e^{2t}-\frac{q(n-2+\mu)k}{(k+1)\mu+nk-2}\,C e^{\hat{\mu}t}+O(e^{4t})\\
x(t)&=& n-2+\mu-\frac{\mu (n-2+\mu)}{\mu +n}\,e^{2t}-\frac{q(n-2+\mu)k}{(k+1)\mu+nk-2}\,C e^{\hat{\mu}t}+O(e^{4t}),
\end{eqnarray*}
and using \eqref{yC},
\[
y(t)=C e^{\hat{\mu}t}\left[1-\frac{\mu (n-2+\mu)}{2k(\mu +n)}\,e^{2t}-\frac{q(n-2+\mu)-[(k+1)\mu +nk-2]}{\hat{\mu}[(k+1)\mu +nk-2]}\,C e^{\hat{\mu}t}+O(e^{4t})\right].
\]
The previous expressions for $x,y$ together with \eqref{inverse0} imply that

\begin{eqnarray*}
[-w(r)]^{q-k}&=&\frac{c_{n,k}}{\lambda}(1+r^2)^{\frac{\mu}{2}}r^{-(2k+\mu-2)}x(\ln r)[y(\ln r)]^k\\
&\rightarrow& \frac{c_{n,k}}{\lambda}(n-2+\mu)C^k:=[-w_0]^{q-k}\;\; (r\rightarrow 0).
\end{eqnarray*}
On the other hand, differentiating the function in \eqref{inverse0} with respect to $r$, we have
\begin{eqnarray*}
w'(r)&=&\frac{1}{q-k}\frac{w(r)}{r}\left [-2+\frac{\mu}{1+e^{2t}}-(q-k)y(t)-r\frac{h'(r)}{h(r)}\right]\\
&=&-w(r)\frac{y(\ln r)}{r}\\
&\rightarrow& 0\;\; (r\rightarrow 0).
\end{eqnarray*}
Hence the function $w$ defined by \eqref{inverse0} is the unique solution of problem \eqref{IVPw} by the first statement of this proposition. This concludes the proof.
\end{proof}

Now we are in position to construct a singular solution of $(P_{\lambda})$. We begin with the following technical lemma.

\begin{lemma}\label{S3L1}
Suppose that $q>q^*(k,\mu-2)$. Then there exists a $t_0\in\RR$ such that ($MS_{q,\mu}$) admits a solution $x(t),y(t)\in C^1(-\infty,t_0)$ satisfying 
\[
(x(t),y(t))\to(\hat{x},\hat{y})\ \textrm{as}\ t\to -\infty.
\]
\begin{proof}
Let $\bar{x}(t):=x(t)-\hat{x}$, $\bar{y}(t):=y(t)-\hat{y}$, and $\bar{X}(t):=(\bar{x}(t),\bar{y}(t))^T$.
Then, by \eqref{perturba}, $\bar{X}$ satisfies
\[
\frac{d}{dt}\bar{X}=A_0\bar{X}+\cS(t,\bar{X}),
\]
where
\begin{align*}
A_0&:=
\left(
\begin{array}{cc}
n-2+\mu-2\hat{x}-q\hat{y} & -q\hat{x}\\
\frac{\hat{y}}{k} & \frac{\hat{x}}{k}+2\hat{y}-\frac{n-2k}{k}
\end{array}
\right)
=
\left(
\begin{array}{cc}
-\hat{x} & -q\hat{x}\\
\frac{\hat{y}}{k} & \hat{y}
\end{array}
\right),\\
\cS(t,\bar{X})&:=
\left(
\begin{array}{c}
\bar{S}(t,\bar{X})\\
\bar{W}(\bar{X})
\end{array}
\right),\\
\bar{S}(t,\bar{X})&:=-\bar{x}^2-q\bar{x}\bar{y}+(\bar{x}+\hat{x})\frac{\mu e^{2t}}{1+e^{2t}},\\
\bar{W}(\bar{X})&:=\frac{\bar{x}\bar{y}}{k}+\bar{y}^2.
\end{align*}
Let $\XX:=C((-\infty,t_0),\RR^2)$ and let $\e>0$ be small.
Here $t_0$ and $\e$ will be chosen later.
We define the ball $B_{\e}:=\left\{X\in\XX:\ \left\|X\right\|_{\XX}:=\sup\{ \left\|X(t)\right\|_{\RR^2}:\, t\in (-\infty, t_0)\}<\e\right\}$.
We show that the Lipschitz constants of $\bar{S}$ and $\bar{W}$ are small.
Let $X_1:=(x_1,y_1)^T, X_2:=(x_2,y_2)^T\in B_{\e}$, and $O:=(0,0)^T$.
When $t$ is large and negative, we have $\mu e^{2t}/(1+e^{2t})<\e$, and whence
\begin{align}
|\bar{S}(t,X_1)-\bar{S}(t,X_2)|
&\le |x_1-x_2||x_1+x_2|+q(|x_1-x_2||y_1|+|y_1-y_2||x_2|)+|x_1-x_2|\frac{\mu e^{2t}}{1+e^{2t}}\nonumber\\
&\le C\e(|x_1-x_2|+|y_1-y_2|),\label{S3L1E1}\\
|\bar{W}(X_1)-\bar{W}(X_2)|
&\le \frac{1}{k}(|x_1-x_2||y_1|+|y_1-y_2||x_2|)+|y_1-y_2||y_1+y_2|\nonumber\\
&\le C\e(|x_1-x_2|+|y_1-y_2|),\label{S3L1E2}\\
|\bar{S}(t,O)|&=|\hat{x}|\frac{\mu e^{2t}}{1+e^{2t}},\label{S3L1E3}\\
|\bar{W}(O)|&=0.\label{S3L1E4}
\end{align}
By $\cF(\bar{X}(t))$ we define
\[
\cF(\bar{X}(t)):=\int_{-\infty}^te^{(t-\tau)A_0}\cS(\tau,\bar{X}(\tau))d\tau.
\]
We find a solution of the equation $\bar{X}(t)=\cF(\bar{X}(t))$ in $B_{\e}$ if $\e>0$ is small and $t_0$ is large and negative. As seen in Section \ref{LinSPoints}, $A_0$ has two eigenvalues with negative real parts when $q>q^*(k,\mu-2)$. Therefore, there exists an $\alpha>0$ such that $\left\|e^{tA_0}\right\|_{\cL(\RR^2,\RR^2)}\le Ce^{-\alpha t}$. If $t<t_0$, then by (\ref{S3L1E1}) and (\ref{S3L1E2}), we have
\begin{align*}
\left\|\cF (X_1(t))-\cF (X_2(t))\right\|_{\RR^2}
&\le\int_{-\infty}^t\left\|e^{(t-\tau)A_0}(\cS (\tau,X_1(\tau))-\cS (\tau,X_2(\tau)))\right\|_{\RR^2}d\tau\\
&\le \int_{-\infty}^t Ce^{-\alpha(t-\tau)}d\tau C\e \left\|X_1-X_2\right\|_{\XX}\\
&\le\frac{C\e}{\alpha}\left\|X_1-X_2\right\|_{\XX}.
\end{align*}
We can now choose $\e>0$ sufficiently small if $t_0$ is large and negative. Therefore, there exists a $t_0\in\RR$ such that
\[
\left\|\cF (X_1)-\cF(X_2)\right\|_{\XX}\le\frac{1}{2}\left\|X_1-X_2\right\|_{\XX}\ \textrm{for}\ X_1,X_2\in B_{\e}.
\]
Now from (\ref{S3L1E3}) and (\ref{S3L1E4}) we see that $\left\|\cS(t,O)\right\|_{\RR^2}=o(1)$ as $t\to -\infty$.
Thus, $\left\|\cF (O)\right\|_{\XX}=o(1)$ as $t\to -\infty$.
We have
\begin{equation}\label{S3L1E5}
\left\|\cF (\bar{X})\right\|_{\XX}\le
\left\|\cF (\bar{X})-\cF (O)\right\|_{\XX}+\left\|\cF (O)\right\|_{\XX}\le \frac{1}{2}\e+o(1)<\e
\end{equation}
provided that $t_0$ is large and negative.
Hence $\cF$ is a contraction mapping on $B_{\e}$.
It follows from the contraction mapping theorem that $\cF$ has a unique fixed point in $B_{\e}$ which is a solution of $\bar{X}(t)=\cF (\bar{X}(t))$.
When $t_0$ is large and negative, $\e>0$ can be chosen arbitrarily small.
By (\ref{S3L1E5}) and the uniqueness of the fixed point in $B_{\e}$, we conclude that $\left\|\bar{X}(t)\right\|_{\RR^2}\to 0$ as $t\to -\infty$.
Thus, $\bar{X}(t)$ is the desired solution.
\end{proof}
\end{lemma}

In the following lemma we define in terms of the orbit $(x(t),y(t))$ obtained in Lemma \ref{S3L1} the value $\tilde{\lambda}$ associated with a singular solution $\tilde{u}(r)$. 

\begin{lemma}\label{S3L2}
Suppose that $q>q^*(k,\mu-2)$.
There exists $\lambda=\tilde{\lambda}>0$ such that the problem 
\begin{equation}\label{S3L2E1}
\begin{cases}
(r^{n-k}(u'(r))^k)'=r^{n-1}c_{n,k}^{-1}\lambda \frac{r^{\mu-2}}{(1+r^2)^{\frac{\mu}{2}}}(1-u(r))^q, & 0<r<1,\\
u(r)<0, & 0<r<1,\\
u(1)=0,
\end{cases}
\end{equation}
has a singular solution $\tilde{u}(r)$ that satisfies
\begin{equation}\label{S3L2E12}
\tilde{u}(r)=-\left[\frac{c_{n,k}\hat{x}\hat{y}^k}{\tilde{\lambda}}\right]^{\frac{1}{q-k}}r^{-\frac{2k-2+\mu}{q-k}}(1+o(1))\ \textrm{as}\ r\to 0.
\end{equation}
\begin{proof}
Let $(x(t),y(t))$ be given by Lemma~\ref{S3L1} and set $h(r) =\frac{r^{\mu-2}}{(1+r^2)^{\frac{\mu}{2}}}$. Since $(x(t),y(t))\to(\hat{x},\hat{y})$ as $t\to-\infty$, (\ref{inverse0}) yields
\begin{align}
u(r) & :=1+w(r)\nonumber\\
&=1-\left[c_{n,k}^{-1}\lambda r^{2k}h(r)\right]^{-\frac{1}{q-k}}\left(x(t)y(t)^k\right)^{\frac{1}{q-k}}\label{S3L2E1+}\\
&\to-\infty\ \textrm{as}\ r\to 0.\nonumber
\end{align}
Then $u(r)$ is a singular solution of \eqref{S3L2E1}.
Differentiating $w(r)$ with respect to $r$, we have
\begin{align*}
(c_{n,k}^{-1}\lambda)^{\frac{1}{q-k}} w'(r) &=\left[r^{2k}h(r)\right]^{-\frac{1}{q-k}}r^{-1}\frac{1}{q-k}(x(t)y(t)^k)^{\frac{1}{q-k}}
\left\{ 2k+\frac{rh'(r)}{h(r)}-\left(\frac{x'(t)}{x(t)}+k\frac{y'(t)}{y(t)}\right)\right\}\\
&=r^{-\frac{q+k}{q-k}}h(r)^{-\frac{1}{q-k}}\frac{1}{q-k}(x(t)y(t)^k)^{\frac{1}{q-k}}\left\{\frac{rh'(r)}{h(r)}+2-\frac{\mu}{1+e^{2t}}+(q-k)y(t)\right\}.
\end{align*}
On the other hand, $|rh'(r)/h(r)|\le C$ for small $r>0$, we have
\begin{align}
|r^{n-k}(w'(r))^k|
&\le r^{n-k}r^{-k\frac{q+k}{q-k}}Cr^{-\frac{k(\mu-2)}{q-k}}\nonumber\\
&=Cr^{\frac{q(n-2k)-k(n-2+\mu)}{q-k}}\to 0\ \textrm{as}\ r\to 0.\label{S3L2E2}
\end{align}
Now, $(MS_{q,\mu})$ has orbits on the $x$-axis and $y$-axis, the uniqueness of a solution of $(MS_{q,\mu})$ shows that $(x(t),y(t))$ is not on the $x$-axis nor on the $y$-axis. Hence $x(t)>0$ and $y(t)>0$ as long as the solution $(x(t),y(t))$ exists.
By (\ref{inverse0}) we see that
\begin{equation}\label{S3L2E3}
w(r)<0.
\end{equation}
Integrating the equation in (\ref{S3L2E1}) over $[s,r]$, we have
\[
r^{n-k}(w'(r))^k-s^{n-k}(w'(s))^k=\int_s^r\tau^{n-1}c_{n,k}^{-1}\lambda h(\tau)(-w(\tau))^qd\tau.
\]
Letting $s\to 0$, (\ref{S3L2E2}) yields
\begin{equation}\label{S3L2E4}
r^{n-k}(w'(r))^k=\int_0^r\tau^{n-1}c_{n,k}^{-1}\lambda h(\tau)(-w(\tau))^qd\tau.
\end{equation}

By simple calculation, we see that the integrand is integrable near $0$.

Next we define
\[
\bar{r}:=\sup\{\delta>0:\ \textrm{The solution $w(r)$ of (\ref{Eq:IVP:0}) exists for $0<r<\delta$}\}.
\]
We show by contradiction that $\bar{r}=\infty$.

Suppose to the contrary, i.e., $\bar{r}<\infty$.
Then, by (\ref{S3L2E3}) and (\ref{S3L2E4}), we see that $(w'(r))^k>0$.
Since $w(r)\to -\infty$ as $r\to 0$, there exists an $r_0>0$ such that $w'(r_0)>0$.
Since $w'$ is continuous and $w$ cannot be $0$, we see that
\begin{equation}\label{S3L2E5}
\textrm{$w'(r)>0$\;\; for\;\; $ 0<r<\bar{r}$.}
\end{equation}
Integrating again the equation in (\ref{S3L2E1}), but now over $[r_0,\bar{r}]$, we have
\begin{eqnarray*}
r^{n-k}(w'(r))^k&=&r_0^{n-k}(w'(r_0))^k+\int_{r_0}^r\tau^{n-1}c_{n,k}^{-1}\lambda h(\tau)(-w(\tau))^qd\tau\\
&>&r_0^{n-k}(w'(r_0))^k.
\end{eqnarray*}
Thus, for all $r_0<r<\bar{r}$, we have
\[
w'(r)>\left(\frac{r_0}{r}\right)^\frac{n-k}{k}w'(r_0)>0.
\]
A last integration gives
\[
0>w(r)>w(r_0)+\int_{r_0}^r\left(\frac{r_0}{\tau}\right)^\frac{n-k}{k}w'(r_0)\,d\tau>-\infty.
\]
Hence $\lim_{r\uparrow\bar{r}}w(r)$ exists. On the other hand, differentiating the function in \eqref{inverse0} with respect to $r$, we have
\begin{eqnarray*}
w'(r)&=&\frac{1}{q-k}\frac{w(r)}{r}\left [-2+\frac{\mu}{1+e^{2t}}-(q-k)y(t)-r\frac{h'(r)}{h(r)}\right]\\
&=&-w(r)\frac{y(\ln r)}{r}.
\end{eqnarray*}
Since $(\hat{x},\hat{y})\in G_-$ and all solutions starting in $G_-$ remain in $G_-$ by Lemma \ref{inward}, we conclude that $\lim_{r\uparrow\bar{r}}w'(r)$ exists and $w'(\bar{r})>0$. Since $w(\bar{r})<0$, $w'(\bar{r})\neq 0$, and $w$ satisfies (\ref{Eq:IVP:0}), $w(r)$ can be locally defined as the solution of (\ref{Eq:IVP:0}) in a right neighborhood of $r=\bar{r}$. This contradicts the definition of $\bar{r}$, and therefore $\bar{r}=\infty$.

Since $w(r)<0$ and $w'(r)>0$ for all $r>0$,
(\ref{newtrans0}) yields that $(x(t),y(t))$ can be defined for all $t\in\RR$.
We now define $\tilde{\lambda}:=2^{\mu/2}c_{n,k}x(0)y(0)^k$.
Then $w(1)=-1$ and $u(1)=0$.
This solution $u(r)$ of (\ref{S3L2E1}) with $\lambda=\tilde{\lambda}$ is denoted by $\tilde{u}(r)$.
Then (\ref{S3L2E12}) follows from (\ref{S3L2E1+}) and  $(\tilde{\lambda},\tilde{u}(r))$ is a desired singular solution.
\end{proof}
\end{lemma}

\section{Intersection number}\label{InterNumber}
In this section we study the intersection number between a regular and a singular solution of suitable equations. These results will be used in the next section to prove Theorem \ref{S4L4} on the multiplicity of solutions of problem $(P_{\lambda})$. 

\medbreak

Let $\tilde{\lambda}$ be as in Lemma \ref{S3L2} and consider the problem
\begin{equation}\label{S4E1}
\begin{cases}
(r^{n-k}(U')^k)'=c_{n,k}^{-1}\tilde{\lambda}r^{n+\mu-3}(-U)^q, & r>0,\\
U(0)=-1,\\
U'(0)=0.
\end{cases}
\end{equation}
Let 
\[
\tilde{U}(r):=-\left[\frac{c_{n,k}\hat{x}\hat{y}^k}{\tilde{\lambda}}\right]^{\frac{1}{q-k}}r^{-\frac{1}{\gamma}},
\]
where $\gamma:=\hat{y}^{-1}=(q-k)/(2k+\mu-2)$.  It is easy to see that $\tilde{U}(r)$ is a singular solution of the first equation in (\ref{S4E1}). Moreover, this equation is of the Emden-Fowler type, which corresponds to the system $(LVS_{q,\rho_-})$ with $h(r)=r^{\mu-2}$. In this case, 
replacing the stationary point $(x,y)=(\hat{x},\hat{y})$ and $\lambda=\tilde{\lambda}$ in \eqref{inverse0}, we obtain the function $\tilde{U}(r)$. Compare with the value $\tilde{\lambda}$ defined in \cite[Theorem 3.1]{SaVe17} and its corresponding singular solution.

The following lemma shows that the singular solution $\tilde{U}(r)$ crosses infinitely many times the regular solution of (\ref{S4E1}).
\begin{proposition}\label{S4P1}
Suppose that $q^*(k,\mu-2)<q<q_{JL}(k,\mu-2)$. Let $U(r)$ be the unique solution of (\ref{S4E1}).
Then
\[
\cZ_{(0,\infty)}[\tilde{U}(\,\cdot\,)-U(\,\cdot\,)]=\infty,
\]
where $\cZ_{I}[\varphi(\,\cdot\,)]$ denotes the number of the zeros of the function $\varphi(\cdot)$ in the interval $I\subset\mathbb{R}$, i.e., $\cZ_{I}[\varphi(\,\cdot\,)]:=\sharp\{r\in I:\ \varphi(r)=0\}$.
\end{proposition}
\begin{proof}
By the local analysis at the point $(\hat{x},\hat{y})$ (see \cite[Section 6]{SaVe17}), we see that this point is a stable spiral for $q^*(k,\mu-2)<q<q_{JL}(k,\mu-2)$. The orbit $(x(t),y(t))$ of $(LVS_{q,n+\mu-2})$ starts from the point $(n+\mu-2,0)$ at $t=-\infty$ and rotates around the point $(\hat{x},\hat{y})$ counterclockwise. Therefore there exists a sequence $\{t_n\}_{n=1}^{\infty}$ such that $t_1<t_2<\cdot\cdot\cdot, y(t_n)=\hat{y}$ for all $n$ and $x(t_2)<x(t_4)<\cdot\cdot\cdot <\hat{x}<\cdot\cdot\cdot <x(t_3)<x(t_1)$. Let $r_n:=e^{t_n}$. By \eqref{inverse0} with $h(r)=r^{\mu-2}$, we have
\begin{eqnarray*}
\frac{\tilde{U}(r_n)}{U(r_n)}&=&\left[\frac{\hat{x}\hat{y}^k}{x(t_n)(y(t_n))^k}\right]^{\frac{1}{q-k}}\\
&=&\left[\frac{\hat{x}}{x(t_n)}\right]^{\frac{1}{q-k}}\begin{cases}
<1, & \mbox{if}\; n\in\{1,3,\cdot\cdot\cdot\},\\
>1, & \mbox{if}\; n\in\{2,4,\cdot\cdot\cdot\},
\end{cases}
\end{eqnarray*}
 and therefore $\cZ_{(0,\infty)}[\tilde{U}(\,\cdot\,)-U(\,\cdot\,)]=\infty$.
\end{proof}

\begin{lemma}\label{S4L1}
Let $\tilde{u}(r)$ be the singular solution obtained in Lemma \ref{S3L2}. Define $\tilde{w}(r)=\tilde{u}(r)-1$ and $(\cF_{\alpha}\tilde{w})(r) =\frac{1}{\alpha}\,\tilde{w}(\frac{r}{\alpha^{\gamma}})$ for $r>0$ and $\alpha>0$.
Then
\begin{equation}\label{S4L1E0}
(\cF_{\alpha}\tilde{w})(r)\to\tilde{U}(r)\ \textrm{in}\ C_{loc}(0,\infty)\ \textrm{as}\ \alpha\to\infty.
\end{equation}

\begin{proof}
Let $I\subset(0,\infty)$ be an arbitrary compact interval. From Lemma \ref{S3L2} we see that
\[
\tilde{w}(r)=-\left[\frac{c_{n,k}\hat{x}\hat{y}^k}{\tilde{\lambda}}\right]^{\frac{1}{q-k}}r^{-\frac{1}{\gamma}}(1+\theta(r)),
\]
where $\theta(r)$ satisfies $\limsup_{r\to 0}\theta(r)=0$.
Therefore
\[
\theta\left(\frac{r}{\alpha^{\gamma}}\right)\to 0\ \textrm{in}\ C(I)\ \textrm{as}\ \alpha\to\infty.
\]
Using this convergence, we have
\begin{align*}
\frac{1}{\alpha}\,\tilde{w}\left(\frac{r}{\alpha^{\gamma}}\right)
&=-\frac{1}{\alpha}\left[\frac{c_{n,k}\hat{x}\hat{y}^k}{\tilde{\lambda}}\right]^{\frac{1}{q-k}}\left(\frac{r}{\alpha^{\gamma}}\right)^{-\frac{1}{\gamma}}\left(1+\theta\left(\frac{r}{\alpha^{\gamma}}\right)\right)\\
&=-\left[\frac{c_{n,k}\hat{x}\hat{y}^k}{\tilde{\lambda}}\right]^{\frac{1}{q-k}}r^{-\frac{1}{\gamma}}\left(1+\theta\left(\frac{r}{\alpha^{\gamma}}\right)\right)\\
&\to \tilde{U}(r)\ \textrm{in}\ C(I)\ \textrm{as}\ \alpha\to\infty.
\end{align*}
Since $I$ was chosen arbitrarily, \eqref{S4L1E0} holds.
\end{proof}
\end{lemma}

\begin{lemma}\label{S4L2}
Let $w(r,\alpha)$ be the solution of the problem
\begin{equation}\label{S4L2E-1}
\begin{cases}
(r^{n-k}(w')^k)'=r^{n-1}c_{n,k}^{-1}\tilde{\lambda}\frac{r^{\mu-2}}{(1+r^2)^{\mu/2}}(-w)^q, & r>0,\\
w(0,\alpha)=-\alpha,\\
w_r(0,\alpha)=0.
\end{cases}
\end{equation}
Let $(\cF_{\alpha}w)(r,\alpha):=\frac{1}{\alpha}\,w(\frac{r}{\alpha^{\gamma}},\alpha)$.
Then
\begin{equation}\label{S4L2E0}
(\cF_{\alpha}w)(r,\alpha)\to U(r)\ \textrm{in}\ C_{loc}[0,\infty)\ \textrm{as}\ \alpha\to\infty.
\end{equation}
\begin{proof}
Let $I\subset [0,\infty)$ be an arbitrary compact interval including $0$.
Let $\bar{w}(r,\alpha):=(\cF_{\alpha}w)(r,\alpha)$.
Then $\bar{w}(r,\alpha)$ satisfies
\[
\begin{cases}
(r^{n-k}(\bar{w}')^k)'=r^{n-1}c_{n,k}^{-1}\tilde{\lambda}\frac{r^{\mu-2}}{(1+\alpha^{-2\gamma}r^2)^{\mu/2}}(-\bar{w})^q, & r>0,\\
\bar{w}(0,\alpha)=-1,\\
\bar{w}_r(0,\alpha)=0.
\end{cases}
\]
Since $-\alpha\le w(r,\alpha)\le 0$ for $r\ge 0$, we see that $-1\le\bar{w}(r,\alpha)\le 0$ for $r\ge 0$.
In particular, $\bar{w}$ is uniformly bounded in $I$.
Integrating the first equation above over $[0,r]$, we have
\begin{equation}\label{S4L2E1}
r^{n-k}(\bar{w}')^k=\int_0^rc_{n,k}^{-1}\tilde{\lambda}\frac{s^{n+\mu-3}}{(1+\alpha^{-2\gamma}s^2)^{\mu/2}}(-\bar{w})^qds.
\end{equation}
Then
\begin{align*}
|\bar{w}_r(r,\alpha)|
&\le\left(r^{-n+k}\int_0^rc_{n,k}^{-1}\tilde{\lambda}s^{n+\mu-3}ds\right)^{1/k}\\
&\le\left(\frac{c_{n,k}^{-1}\tilde{\lambda}}{n+\mu-2}\right)^{1/k}r^{\frac{k+\mu-2}{k}}.
\end{align*}
Here $\bar{w}(r,\alpha)$ is equicontinuous in $I$.
By the Ascoli-Arzel\`a theorem, there exist a sequence $\{\alpha_j\}$ diverging to $+\infty$ and $\bar{w}^*(r)\in C(I)$ such that
\begin{equation}\label{S4L2E2}
\bar{w}(r,\alpha_j)\to\bar{w}^*(r)\ \textrm{in}\ C(I)\ \textrm{as}\ j\to\infty.
\end{equation}
By (\ref{S4L2E1}),
\[
\bar{w}(r,\alpha_j)=-1+\int_0^r\left( t^{-n+k}\int_0^tc_{n,k}^{-1}\tilde{\lambda}\frac{s^{n+\mu-3}}{(1+\alpha_j^{-2\gamma}s^2)^{\mu/2}}(-\bar{w}(s,\alpha_j))^qds\right)^{1/k}dt.
\]
Letting $j\to\infty$, we have
\begin{equation}\label{S4L2E3}
\bar{w}^*(r)=-1+\int_0^r\left(t^{-n+k}\int_0^tc_{n,k}^{-1}\tilde{\lambda}s^{n+\mu-3}(-\bar{w}^*(s))^qds\right)^{1/k}dt\;\; \textrm{for}\ r\in I,
\end{equation}
since the following two convergences are uniform on $I$:
\[
c_{n,k}^{-1}\tilde{\lambda}\frac{s^{n+\mu-3}}{(1+\alpha_j^{-2\gamma}s^2)^{\mu/2}}(-\bar{w}(s,\alpha_j))^q\to c_{n,k}^{-1}\tilde{\lambda}s^{n+\mu-3}(-\bar{w}^*(s))^q\ \textrm{and}
\]
\[
\left(t^{-n+k}\int_0^tc_{n,k}^{-1}\tilde{\lambda}\frac{s^{n+\mu-3}}{(1+\alpha_j^{-2\gamma}s^2)^{\mu/2}}(-\bar{w}(s,\alpha_j))^qds\right)^{1/k} \to
\left(t^{-n+k}\int_0^tc_{n,k}^{-1}\tilde{\lambda}s^{n+\mu-3}(-\bar{w}^*(s))^qds\right)^{1/k}.
\]
The equality (\ref{S4L2E3}) indicates both that $\bar{w}^*(r)\in C^2(I^{i})\cap C^1(I)$ and $\bar{w}^*(r)$ is the solution of the problem
\[
\begin{cases}
(r^{n-k}((\bar{w}^*)')^k)'=c_{n,k}^{-1}\tilde{\lambda}r^{n+\mu-3}(-\bar{w}^*)^q, & r\in I,\\
\bar{w}^*(0,\alpha)=-1,\\
\bar{w}_r^*(0,\alpha)=0,
\end{cases}
\]
where $I^{i}$ denotes the set of the interior points of $I$.
Therefore $\bar{w}^*(r)=U(r)$ for $r\in I$.
Since $I$ can be chosen arbitrarily, (\ref{S4L2E2}) implies (\ref{S4L2E0}).
\end{proof}
\end{lemma}

\begin{lemma}\label{S4L3}
Suppose that $q^*(k,\mu-2)<q<q_{JL}(k,\mu-2)$.
Then
\begin{equation}\label{S4L3E0}
\cZ_{[0,1]}[\tilde{w}(\,\cdot\,)-w(\,\cdot\,,\alpha)]\to\infty\ \textrm{as}\ \alpha\to\infty.
\end{equation}
\begin{proof}
Since $q^*(k,\mu-2)<q<q_{JL}(k,\mu-2)$, Proposition~\ref{S4P1} states that
\begin{equation}\label{S4L3E1}
\cZ_{[0,\infty)}[\tilde{U}(\,\cdot\,)-U(\,\cdot\,)]=\infty.
\end{equation}
Let $U_1$, $U_2$ be solutions of the equation in (\ref{S4E1}).
We have
\begin{multline}\label{S4L3E2}
kr^{n-k}U_2'^{k-1}(U_2-U_1)''+\left\{kr^{n-k}U_1'V_1+(n-k)r^{n-k-1}V_2\right\}(U_2-U_1)'\\
=r^{n+\mu-3}c_{n,k}^{-1}\tilde{\lambda}V_3(U_2-U_1),
\end{multline}
where $V_1$, $V_2$, and $V_3$ are continuous function of $r$.
We set $U_1:=U$ and $U_2:=\tilde{U}$.
Since $U_2'\neq 0$ for $r>0$, the ODE (\ref{S4L3E2}) is of second order.
By the uniqueness of the solution of ODEs, each zero of $\tilde{U}(\,\cdot\,)-U(\,\cdot\,)$ is simple.
The zero set of $\tilde{U}(\,\cdot\,)-U(\,\cdot\,)$ does not have an accumulation point, and hence each zero is isolated.
Because of this fact and (\ref{S4L3E1}), for every large $N>0$, there exists an $R>0$ such that $N+1\le\cZ_{[0,R]}[\tilde{U}(\,\cdot\,)-U(\,\cdot\,)]<\infty$.
By Lemmas~\ref{S4L1} and \ref{S4L2}
\[
\cF_{\alpha}\tilde{w}\to\tilde{U}\ \ \textrm{in}\ \ C_{loc}(0,\infty),
\]
\[
\cF_{\alpha}w\to U\ \ \textrm{in}\ \ C_{loc}(0,\infty),
\]
as $\alpha\to\infty$. Therefore, for each zero of $\tilde{U}-U$, there exists at least one zero of $\cF_{\alpha}\tilde{w}-\cF_{\alpha}w$ in a neighborhood of the zero of $\tilde{U}-U$ when $\alpha$ is large.
We conclude that
\[
\cZ_{[0,R]}[(\cF_{\alpha}\tilde{w})(\,\cdot\,)-(\cF_{\alpha}w)(\,\cdot\,,\alpha)]\ge\cZ_{[0,R]}[\tilde{U}(\,\cdot\,)-U(\,\cdot\,)]-1\ge N,
\]
provided that $\alpha$ is large.
Since $\cZ_{[0,R]}[(\cF_{\alpha}\tilde{w})(\,\cdot\,)-(\cF_{\alpha}w)(\,\cdot\,,\alpha)]=\cZ_{[0,\alpha^{-\gamma}R]}[\tilde{w}(\,\cdot\,)-w(\,\cdot\,,\alpha)]$,
we have
\[
\cZ_{[0,\alpha^{-\gamma}R]}[\tilde{w}(\,\cdot\,)-w(\,\cdot\,,\alpha)]\ge N.
\]
When $\alpha>0$ is large, we have $[0,\alpha^{-\gamma}R]\subset[0,1]$, whence
\[
\cZ_{[0,1]}[\tilde{w}(\,\cdot\,)-w(\,\cdot\,,\alpha)]\ge\cZ_{[0,\alpha^{-\gamma}R]}[\tilde{w}(\,\cdot\,)-w(\,\cdot\,,\alpha)]\ge N
\]
for large $\alpha>0$.
The number $N$ can be chosen arbitrarily large, whence (\ref{S4L3E0}) holds.
\end{proof}
\end{lemma}

\section{Proof of Theorem \ref{S4L4}}
\begin{proof}
Let $w(r,\alpha)$ be the solution of (\ref{S4L2E-1}).
Then $\bar{w}(r,\alpha):=({\tilde{\lambda}}/{\lambda})^{1/(q-k)}w(r,\alpha)$ satisfies
\[
\begin{cases}
(r^{n-k}(\bar{w}')^k)'=r^{n-1}c_{n,k}^{-1}\lambda\frac{r^{\mu-2}}{(1+r^2)^{\mu/2}}(-\bar{w})^q, & r>0,\\
\bar{w}(0,\alpha)=-\left(\frac{\tilde{\lambda}}{\lambda}\right)^{1/(q-k)}\alpha,\\
\bar{w}_r(0,\alpha)=0.
\end{cases}
\]
Further, $-({\tilde{\lambda}}/{\lambda})^{1/(q-k)}\alpha\le\bar{w}(r,\alpha)<0$ for $r\ge 0$, and $\bar{w}_r(0,\alpha)>0$ for $r>0$.
Let $u(r,\alpha)=1+\bar{w}(r,\alpha)$.
Then $u$ satisfies the equation in $(P_{\lambda})$.
If $\alpha>(\lambda/\tilde{\lambda})^{1/(q-k)}$, then $u(0,\alpha)<0$.
Since $u$ is increasing, $u$ is a solution of $(P_{\lambda})$ if and only if $u(1,\alpha)=0$.
This equation is equivalent to
\begin{equation}\label{S4L4E1}
\lambda=\tilde{\lambda}(-w(1,\alpha))^{q-k}.
\end{equation}
Since $0=\tilde{u}(1)=1+\tilde{w}(1)$, we have $\tilde{w}(1)=-1$.
We now study $\cZ_{[0,1]}[\tilde{w}(\,\cdot\,)-w(\,\cdot\,,\alpha)]$, which we call the intersection number.
By Lemma~\ref{S4L3},
\begin{equation}\label{S4L4E2}
\cZ_{[0,1]}[\tilde{w}(\,\cdot\,)-w(\,\cdot\,,\alpha)]\to\infty\ \textrm{as}\ \alpha\to\infty.
\end{equation}
Using the same argument used in the proof of Lemma~\ref{S4L3}, we can easily show that each zero of $\tilde{w}(\,\cdot\,)-w(\,\cdot\,,\alpha)$ in $(0,\infty)$ is simple.
Let $\alpha>0$ be fixed.
Since $\tilde{w}(0)-w(0,\alpha)=\infty$, the zero set is uniformly away from the origin.
The coefficient of the second derivative of the ODE which $\tilde{w}-w$ satisfies is uniformly away from zero on a compact interval in $(0,\infty)$.
Since the zero set does not have an accumulation point, we have $\cZ_{[0,1]}[\tilde{w}(\,\cdot\,)-w(\,\cdot\,,\alpha)]<\infty$. Now
each zero depends continuously on $\alpha$. Therefore the intersection number on $[0,1]$ is preserved if a zero does not go out from the boundary of $[0,1]$ and if another zero does not come from the boundary.
Since $\tilde{w}(0)-w(0,\alpha)=\infty$, a zero cannot go out or come from $r=0$.
By (\ref{S4L4E2}) we see that a zero comes from $r=1$ infinitely many times.
Therefore, there exists a sequence $\{\alpha_n\}_{n=1}^{\infty}$ such that $\alpha_1<\alpha_2<\cdots<\alpha_n<\cdots\to\infty$, $\tilde{w}(1)-w(1,\alpha_n)=0$, and $\cZ_{[0,1]}[\tilde{w}(\,\cdot\,)-w(\,\cdot\,,\alpha)]=n$.
Since $\tilde{w}(1)=-1$, we have $w(1,\alpha_n)=-1$ for every $n\ge 1$.
Now, if $w(1,\alpha)\neq -1$ and $\cZ_{[0,1]}[\tilde{w}(\,\cdot\,)-w(\,\cdot\,,\alpha)]$ is odd, then $w(1,\alpha)<-1$.
On the other hand, if $w(1,\alpha)\neq -1$ and $\cZ_{[0,1]}[\tilde{w}(\,\cdot\,)-w(\,\cdot\,,\alpha)]$ is even, then $w(1,\alpha)>-1$.
Since $w(1,\alpha)$ is continuous in $\alpha$, $w(1,\alpha)$ oscillates around $-1$ infinitely many times as $\alpha\to\infty$.
If $\lambda=\tilde{\lambda}$, then (\ref{S4L4E1}) has infinitely many solutions.
For each $N\ge 1$, there exists an $\e>0$ such that if $|\lambda-\tilde{\lambda}|<\e$, whence (\ref{S4L4E1}) has at least $N$ solutions.
Thus, the conclusion holds.
\end{proof}

\bibliographystyle{plain}
\bibliographystyle{apalike}
\bibliography{kHessianMatukuma}

$\mbox{}$

\noindent {\footnotesize{\bf Yasuhito Miyamoto}, Graduate School of Mathematical Sciences,
The University of Tokyo\\
3-8-1 Komaba,
Meguro-ku,
Tokyo 153-8914,
Japan, Email: miyamoto@ms.u-tokyo.ac.jp}

$\mbox{}$

\noindent {\footnotesize{\bf Justino S\'anchez}, Departamento de Matem\'{a}ticas, Universidad de La Serena\\
 Avenida Cisternas 1200, La Serena, Chile, Email: jsanchez@userena.cl}

$\mbox{}$

\noindent {\footnotesize{\bf Vicente Vergara}*, Departamento de Matem\'{a}ticas, Universidad Cat\'{o}lica del Norte\\
Angamos 0610,
Antofagasta, Chile, Email, vicente.vergara@ucn.cl}

\end{document}